\documentclass[11pt]{article}
\usepackage{amssymb}
\usepackage{amsfonts}
\usepackage{amssymb,latexsym}
\usepackage{amsmath}
\usepackage{graphics}
\usepackage{graphicx}
\usepackage{amsmath,amsfonts,amssymb,amsbsy,amsthm,latexsym,lscape}
\usepackage{url}
\def\oneandahalf{\par\baselineskip=18pt}

\def\E{{\rm E}}
	{\end{list}}
\newenvironment{hangref}{\begin{list}{}{\setlength{\itemsep}{0pt}%
			\setlength{\parsep}{0pt}\setlength{\leftmargin}{+\parindent}%
			\setlength{\itemindent}{-\parindent}}}{\end{list}}

\def\standard{\oddsidemargin=0in
	\evensidemargin=0in
	\topmargin =-.5in
	\textheight=9.0in
	\textwidth=6.5in}

\standard

\newtheorem{theorem}{Theorem}%[section]
\newtheorem{proposition}{Proposition}%[section]
\newtheorem{assumption}{Assumption}%[section]
\begin{document}

\normalsize

\title{
Mean-Variance Portfolio Selection: Statistical Properties and Distributionally Robust Optimization}
\author{
        Zhaolin Hu\\
        {\footnotesize
        School of Economics and
        Management, Tongji University, Shanghai 200092, China}\\
        {\footnotesize 
        russell@tongji.edu.cn}
}

\date{}

\maketitle

\begin{abstract}
We study Markowitz's mean-variance portfolio optimization problem. When practically using this model, the mean vector and the covariance matrix of the assets returns often need to be estimated from the sample data. The sample errors will be propagated to the optimization output. In this paper, we consider three commonly used mean-variance models and build the asymptotic properties for the conventional sample approximations that are widely adopted and studied, by leveraging the stochastic optimization theory. We show that for all three models, under certain conditions the sample approximations have the desired consistency and achieve a convergence rate of square root of sample size, and the asymptotic variance depends on the first four moments of the returns. To address the variability and the uncertainty in the data driven portfolio optimization, we further consider the distributionally robust optimization (DRO) for the mean-variance model. For the ambiguity sets constructed by the $\phi$-divergence and the Wasserstein distance, we build the tractable reformulations for the DRO models. We conduct numerical experiments to test the asymptotic properties for the estimation. We also conduct experiments to illustrate that the asymptotic normality might not hold when the fourth moments of the returns do not exist.
\end{abstract}

\oneandahalf

%\double

\section{Introduction}\label{sec:intro}

The mean-variance model is a classical portfolio selection model proposed by Markowitz (1952). Its basic idea is to use variance to model the portfolio risk and propose to make a tradeoff between return and risk. The idea opened a door for quantitative financial investment and risk management and Harry Markowitz was awarded the 1990 Nobel prize in economics for his portfolio theory. In this model, decision makers optimize one metric with constraint on the other metric, or optimize a weighted sum of the two metrics. When practically using this model, decision makers need to specify the mean vector and covariance matrix for the random returns/losses of the portfolio assets. However, in practice the parameters are rarely known exactly and often need to be estimated. By using the estimators to replace the true parameters one obtains a sample approximation for the true model and the sample errors may be propagated to that of the decisions. In this paper, we consider the sample errors when formulating and solving this model. We build asymptotic properties for variants of mean-variance models. 

While the mean-variance model is easy to solve, the performance of the optimal portfolio returned by the sample approximation typically admits a substantial variability. This issue has been concerned and studied extensively. A vast volume of studies investigated the estimation errors in finance. Britten-Jones (1999) studied the sampling errors in the mean-variance model. They derived inference procedures for the hypotheses about the portfolio weights. DeMiguel et al. (2009b) empirically evaluated the out-of-sample performance of the sample-based mean-variance model against the $1/N$ portfolio and found that the former could not consistently beat the latter over a number of empirical datasets. They showed that a large sample size is required for the sample-based mean-variance strategy to outperform the $1/N$ benchmark. DeMiguel et al. (2009a) and Brodie et al. (2009) studied using regularization to address the sample errors in the mean-variance models. Many of the studies targeted to derive some good estimators for the parameters to achieve better out-of-sample performances for the portfolio optimization model. 

The mean-variance model can be viewed as a representative of the mean-risk portfolio optimization models. Lim et al. (2011) studied a portfolio selection model where the risk is measured by the conditional value-at-risk, and showed that the sample solutions are fragile. Shapiro et al. (2014) used the stochastic optimization theory to build statistical properties for the estimation of various risk measures. Hu and Zhang (2018) studied the statistical properties for computing the utility-based shortfall risk. Wang et al. (2023) studied a mean-variance model that optimizes a weighted sum of the mean and the variance where the portfolio consists of derivative securities. They developed an estimator for the covariance matrix based on simulation and constructed an approximation problem for the mean-variance model. They built statistical properties for the approximation. However, their estimator for the covariance matrix may not be positive semi-definite. They projected the estimator to a positive definite matrix and further solved a new approximation problem. In this paper, we study the sample approximations of the mean-variance models where the mean vector and covariance matrix are estimated by the conventional empirical estimators. This kind of empirical approximations are widely adopted in the financial practice. We analyze the asymptotic properties of the sample approximations.  

The variability of the sample approximation is closely related to the perturbation analysis on the mean vector and the covariance matrix. Note, however, that the variance is not in an expectation expression. Actually, it takes a composite form as in the definition of the variance the expectation is embedded in the outer expectation. Directly conducting a perturbation analysis is not an easy task. In this paper, we treat the variance as an optimal value of an expectation and convert the mean-variance model to a stochastic program that only involves expectations. We then use stochastic optimization theory to build the statistical properties of the sample mean-variance model. Based on this approach, our analysis is applicable for various mean-variance models. We show that the sample approximations have the desired consistency and the convergence rate of the sample approximations is in the order of square root of the sample size. The asymptotic variance is affected by the first four moments of the random distribution of the assets. It suggests that while the mean-variance model only depends on the first two moments of the underlying randomness, the sample approximation performance (e.g., the variation) will be affected by the higher moments of the distribution. Furthermore, our analysis shows that adding regularization to the mean-variance model does not affect the asymptotic regimes. Besides building the theoretical results, our work also shows the powerfulness of the stochastic optimization theory via the analysis, which might provide some inspiration on using stochastic optimization to address other sampling-based optimization problems.

The variability of the sample approximation affects the practical applications of the mean-variance model. In practice, there are often limited data to achieve a high accuracy for the sample approximation. To address the uncertainty of the random distribution of the asset returns, the distributionally robust optimization (DRO) formulation has been adopted for the mean-variance portfolio optimization. In this paper, we also consider the DRO models, in which we study two classes of ambiguity sets modeled by the $\phi$-divergence and the Wasserstein distance. The two metrics are widely studied and used in the DRO literature, see, e.g., the comprehensive review in Kuhn et al. (2025). Blanchet et al. (2022) built an elegant reformulation for the mean-variance DRO model with the Wasserstein distance. They noted that the variance is nonlinear in the probability distribution and spent effort to address the nonlinearity and the worst-case variance. We show that treating the variance as the optimal value of an expectation could be adopted to overcome this obstacle and substantially simplify the analysis of the DRO model. For the $\phi$-divergence we use the conventional convex duality and reformulate the DRO problems as single layer optimization problems. For the Wasserstein distance, we consider a slightly more general DRO model than that of Blanchet et al. (2022) and show its equivalence to a single layer optimization problem. We show that the argument can be adopted to prove the reformulation built in Blanchet et al. (2022).                  

The rest of this paper is organized as follows. In Section \ref{sec:MV} we study a mean-variance model that optimizes the risk with constraint on the return, and build statistical properties for the sample approximation. In Section \ref{sec:MVA} we generalize the analysis for more variants of mean-variance models. In Section \ref{sec:DRO}, we analyze the mean-variance DRO models. We conduct some simulation experiments in Section \ref{sec:num} to test the theoretical findings.  Section \ref{sec:con} concludes the paper.

\section{Mean-Variance Model}\label{sec:MV}

Suppose an investor aims to invest on $k$
assets with random return vector $\xi=(\xi^1,\cdots,\xi^k)^{\rm T}$ where
$\xi^j$ is the return rate of asset $j$. Suppose $x_j$ is the capital
invested in asset $j$. Then the random return of the portfolio is
$H(x, \xi)=x^{\rm T}\xi$ where $x=(x_1,\cdots,x_k)^{\rm T}$. The mean-variance model proposed by Markowitz (1952) suggests to optimally balance the risk and return. There are alternative expressions for the mean-variance model. We first study the following formulation
\begin{eqnarray}\label{MVportfolio}
&\displaystyle\mathop{\mbox{minimize}}_{x\in X} & \mbox{V}\left[H(x, \xi)\right]\\
&\mbox{subject to} &  \E\left[H(x, \xi)\right]
\ge r,\nonumber 
\end{eqnarray}
where $\E$ and $\mbox{V}$ denote the expectation and variance of the random function, $r$ is a prespecified threshold and $X$ is the feasible set of $x$. Without loss of generality, we assume $X$ is a compact convex set throughout this paper. In Problem (\ref{MVportfolio}), the support of $\xi$, denoted as
$\Xi$, is a closed subset of $\Re^k$. Throughout the paper, it is implicitly assumed that $H(x, \xi)$ has finite first and second moments. Therefore, Problem (\ref{MVportfolio}) is well defined. There are other mean-variance variants of Problem (\ref{MVportfolio}). We defer the discussion in the next section. 

For the linear portfolio $H(x, \xi)=x^{\rm T}\xi$, Problem (\ref{MVportfolio}) is actually a quadratic program
\begin{eqnarray}\label{MVportfoliodet}
&\displaystyle\mathop{\mbox{minimize}}_{x\in X} & x^{\rm T}\Sigma x\\
&\mbox{subject to} & \mu^{\rm T} x \ge r, \nonumber \label{MVportfoliodetcons}
\end{eqnarray}
where $\mu$ and $\Sigma$ are the mean vector and covariance matrix of $\xi$. Throughout the paper, we assume that $\Sigma$ is a positive definite matrix. Therefore, the objective function in Problem (\ref{MVportfoliodet}) is strictly convex and Problem (\ref{MVportfoliodet}) has a unique optimal solution. With given parameters $\mu$ and $\Sigma$, the quadratic program is easy to solve. When practically formulating and solving the problem, however, the parameters $\mu$ and $\Sigma$ are not available directly but need to be estimated. Investors often use a sample from $\xi$ to estimate $\mu$ and $\Sigma$. Suppose we have $n$ independent and identically distributed (i.i.d.) observations of $\xi$, denoted as $\left\{\xi_1,\xi_2,\cdots,\xi_n\right\}$, which may be the real historical data or the sample simulated from the distribution of $\xi$. Let
\begin{equation}\label{equ:estimator}
\hat\mu=\frac{1}{n}\sum_{i=1}^n\xi_i,\  \hat\Sigma= \frac{1}{n} \sum_{i=1}^n\left(\xi_i-\hat\mu\right)\left(\xi_i-\hat\mu\right)^{\rm T},
\end{equation}
which are called sample mean and sample covariance matrix. Note that we use $\hat\Sigma$ as an estimator of $\Sigma$. In statistics, the estimator $\hat\Sigma_u=\frac{1}{n-1} \sum_{i=1}^n\left(\xi_i-\hat\mu\right)\left(\xi_i-\hat\mu\right)^{\rm T}$ is also frequently used, which is an unbiased estimator of $\Sigma$.

By using the sample estimators in (\ref{equ:estimator}) to replace $\mu$ and $\Sigma$ in Problem (\ref{MVportfoliodet}), we obtain the following sample approximation problem
\begin{eqnarray}\label{MVsampleportfolio}
&\displaystyle\mathop{\mbox{minimize}}_{x\in X} & x^{\rm T}\hat\Sigma x\\
&\mbox{subject to} & \hat\mu^{\rm T} x \ge r. \nonumber
\end{eqnarray}
Let $v_n$ and $v^*$ denote the optimal values of Problems (\ref{MVsampleportfolio}) and (\ref{MVportfoliodet}) respectively. A natural and central problem in this data driven approach is how well $v_n$ approximates $v^*$. Ban et al. (2018) built that the sample mean-variance model is fragile. They did numerical experiments to check the robustness of the sample based optimal values. Theoretically, they proved that the solution of Problem (\ref{MVsampleportfolio}) converges to that of Problem (\ref{MVportfoliodet}) in probability. We show that the solution converges with probability one (w.p.1). 

There are substantial difficulties for studying the perturbation of the mean vector and covariance matrix. This structure sets obstacles for the use of the stochastic optimization theory. In this paper, we propose to consider a stochastic program reformulation and analyze the statistical properties of the sample approximation. The consistency of the solution and the optimal value is relatively simple under this stochastic optimization perspective.

\subsection{Asymptotic Properties} 

In this paper, we aim to build the asymptotic properties based on the theory of stochastic optimization. This approach can be directly implemented to build the consistency. We show that the convergence rate is in the order of $n^{-1/2}$. We further derive the asymptotic variance. A critical observation is that Problem (\ref{MVportfolio}) can be converted equivalently to the following problem
\begin{eqnarray}\label{MVportfolioExp}
&\displaystyle\mathop{\mbox{minimize}}_{x\in X, t\in\Re} & \E\left[\left(H(x, \xi)-t\right)^2\right]\\
&\mbox{subject to} &  \E\left[H(x, \xi)\right]
\ge r,\label{MVportfolioExpconst}
\end{eqnarray}
and Problem (\ref{MVsampleportfolio}) can be converted equivalently to the following problem 
\begin{eqnarray}\label{portfolioExpsaa}
&\displaystyle\mathop{\mbox{minimize}}_{x\in X, t\in\Re} & \frac{1}{n}\sum_{i=1}^n\left(H(x,\xi_i)-t\right)^2 \\
&\mbox{subject to} &  \frac{1}{n}\sum_{i=1}^n H(x,\xi_i)
\ge r, \nonumber
\end{eqnarray}
 where $\left\{\xi_1,\xi_2,\cdots,\xi_n\right\}$ is the set of i.i.d. sample that is used in (\ref{equ:estimator}). We summarize the results in the following proposition.
\begin{proposition}\label{pro:mvequ}
(i) Problem (\ref{MVportfolioExp}) and Problem (\ref{MVportfolio}) are equivalent, in the sense that $x$ is an optimal solution of Problem
(\ref{MVportfolio}) if and only if there exists $t$ such that
$(x,t)$ is an optimal solution of Problem (\ref{MVportfolioExp}), and both problems have the same optimal value $v^*$. 

(ii) Problem (\ref{portfolioExpsaa}) and Problem (\ref{MVsampleportfolio}) are equivalent, in the sense that $x$ is an optimal solution of Problem
(\ref{MVsampleportfolio}) if and only if there exists $t$ such that
$(x,t)$ is an optimal solution of Problem (\ref{portfolioExpsaa}), and both problems have the same optimal value $v_n$.
\end{proposition}
\noindent\textbf{Proof.} (i) The equivalence is built on the fact that
\begin{equation}\label{varsp}
\mbox{V}\left[H(x,\xi)\right]=\inf_{t\in \Re}\ \E\left[\left(H(x,\xi)-t\right)^2\right].
\end{equation}
To see this, consider the objective function for the infimum in (\ref{varsp}), which can be expressed as
\begin{equation}
\E\left[\left(H(x, \xi)-t\right)^2\right]=t^2-2\E\left[H(x,\xi)\right]t+\E\left[\left(H(x,\xi)\right)^2\right]. \nonumber
\end{equation}
It is a strictly convex function and has a unique minimizer $t^*(x)=\E\left[H(x,\xi)\right]$. Plugging $t^*(x)$ into the objective function yields that the infimum is the variance of $H(x,\xi)$. Integrating the minimization and infimum we immediately obtain the equivalence.

(ii) The equivalence can be built based on the argument for (i) by replacing the distribution with its empirical distribution.  
\hfill$\Box$

The perspective of treating the variance and the mean of a random variable as the optimal value and the optimal solution of an expectation function is dated back to early analysis, e.g., Huber (1964). Proposition \ref{pro:mvequ} suggests that analyzing the approximation of Problem  (\ref{MVsampleportfolio}) to Problem (\ref{MVportfoliodet}) is equivalent to analyzing the approximation of Problem (\ref{portfolioExpsaa}) to Problem (\ref{MVportfolioExp}). Problem (\ref{MVportfolioExp}) is a standard stochastic optimization problem. The more interesting fact is that Problem (\ref{portfolioExpsaa}) is exactly the sample average approximation (SAA) of Problem (\ref{MVportfolioExp}). This bridge enables us to use the stochastic optimization theory to build the statistical properties of the sample mean-variance approximation. 

When $H(x,\xi)$ is a linear function, which is the focus of this paper, we can see that the objective function in Problem (\ref{MVportfolioExp}) is jointly convex in $(x,t)$. For the more general nonlinear function, the convexity may not be guaranteed even when $H(x,\xi)$ is convex in $x$. A partial reason is that the quadratic function is not monotone. However, for the general function, even with the absence of convexity, the reformulation allows us to solve the mean-variance model and conduct analysis by using stochastic optimization techniques. In this paper, we focus on the statistical properties for the linear portfolio. 

Note that $\Sigma$ is positive definite. Problem (\ref{MVportfoliodet}) has a unique optimal solution. Moreover, $\E\left[\left(H(x, \xi)-t\right)^2\right]$ is strictly convex in $t$. Therefore, Problem (\ref{MVportfolioExp}) has a unique optimal solution. We denote it as $\left(x^*,t^*\right)$. 
We make the following assumptions.
\begin{assumption}\label{assu:Slater}
There exists a point $x\in X$ such that $\E\left[H(x, \xi)\right]> r$.
\end{assumption}
Assumption \ref{assu:Slater} is a Slater condition, which is a standard constraint qualification in optimization.

\begin{assumption}\label{assu:lip}
There exists a random function $M(\xi)$ with
$\E\left[M(\xi)\right]<\infty$ such that
\[|H(x,\xi)|\leq M(\xi), \forall\ x\in X,\ a.e. \ \xi\in \Xi.\]
Furthermore, there exists a random function $K(\xi)$ with
$\E\left[K(\xi)^2\right]<\infty$ such that
\[\left|H(x_1,\xi)-H(x_2,\xi)\right|\le K(\xi)\
\|x_1-x_2\|, \forall\ x_1,x_2\in X,\ a.e.\ \xi\in \Xi.\]
\end{assumption}
Assumption
\ref{assu:lip} requires that $H(x,\xi)$ is dominated by an
integrable random variable and satisfies the Lipschitz condition. It
is a standard assumption in stochastic optimization literature, see,
for instance, Shapiro et al. (2014). It is also a critical assumption for the
well-definedness and the differentiability of the expectations that
are considered below.

In both Problem (\ref{MVportfolioExp}) and Problem (\ref{portfolioExpsaa}), $t\in\Re$. In this paper, we assume that we can specify a compact interval $T$ such that for the optimal solution $(x,t)$ of Problem  (\ref{MVportfolioExp}), $t$ is an interior point of $T$. To analyze the asymptotic properties, we replace $t\in\Re$ with $t\in T$ in Problem (\ref{MVportfolioExp}) and Problem (\ref{portfolioExpsaa}). We also make a parallel assumption on $\left(H(x, \xi)-t\right)^2$.
\begin{assumption}\label{assu:lip:variance}
There exists a random function $\tilde M(\xi)$ with
$\E\left[\tilde M(\xi)\right]<\infty$ such that
\[|\left(H(x,\xi)-t\right)^2|\leq \tilde M(\xi), \forall\ (x,t)\in X\times T,\ a.e.\ \xi\in \Xi.\]
Furthermore, there exists a random function $\tilde K(\xi)$ with
$\E\left[\tilde K(\xi)^2\right]<\infty$ such that
\[\left|\left(H(x_1,\xi)-t_1\right)^2-\left(H(x_2,\xi)-t_2\right)^2\right|\le \tilde K(\xi)\
\|(x_1,t_1)-(x_2,t_2)\|, \forall\ (x_1,t_1), (x_2,t_2)\in X\times T, \ a.e. \ \xi\in \Xi.\]
\end{assumption}
Recall that $\left(x^*,t^*\right)$ is the unique optimal solution of Problem (\ref{MVportfolioExp}). We write the constraint (\ref{MVportfolioExpconst}) as $\E\left[-H(x, \xi)+r\right]
\le 0$ and associate the constraint with a Lagrangian multiplier $\lambda$. Let $\Lambda$ denote the set of optimal multipliers associated with $\left(x^*,t^*\right)$. Let $\left(x_n, t_n\right)$ denote an optimal solution of Problem (\ref{portfolioExpsaa}). Define the random variable 
\begin{equation}
Y(x,t,\lambda)= \lim_{n\to \infty} \sqrt{n}\left[\frac{1}{n}\sum_{i=1}^n\left[\left(H(x,\xi_i)-t\right)^2-\lambda H(x,\xi_i)\right]-\E\left[\left(H(x,\xi)-t\right)^2-\lambda H(x,\xi)\right]\right]. \nonumber
\end{equation}
By the central limit theorem (CLT), $Y(x,t,\lambda)$ has a normal distribution $N(0,\sigma^2(x,t,\lambda))$ with $\sigma^2(x,t,\lambda)=\mbox{V}\left[\left(H(x,\xi)-t\right)^2-\lambda H(x,\xi)\right]$. We build the following theorem. 
\begin{theorem}\label{thm:asyoptval}
Suppose that Assumptions \ref{assu:Slater} and \ref{assu:lip} are satisfied. Then $v_n \to v^*$ with probability one (w.p.1), and $(x_n,t_n)\to (x^*,t^*)$ w.p.1. Suppose that Assumptions \ref{assu:Slater}, \ref{assu:lip} and \ref{assu:lip:variance} are satisfied. Then 
\begin{equation*}
\sqrt{n} \left(v_n-v^*\right) \Rightarrow \sup_{\lambda\in \Lambda}\ \left[Y(x^*,t^*, \lambda)\right],   
\end{equation*}
where ``$\Rightarrow$'' denotes the convergence in distribution. If, moreover, $\Lambda=\left\{\lambda^*\right\}$ is a singleton, then
\begin{equation*}
\sqrt{n} \left(v_n-v^*\right) \Rightarrow N(0,\sigma^2),
\end{equation*}
where
\begin{equation}\label{CLTVvmin}
\sigma^2=\mbox{V}\left[\left(H(x^*,\xi)-\E\left[H(x^*,\xi)\right]\right)^2-\lambda^* H(x^*,\xi)\right].
\end{equation}
\end{theorem}
\noindent\textbf{Proof.} As discussed we can impose a compact interval $T$ and use $t\in T$ to replace $t\in\Re$, and consider the compact set $X\times T$. For the i.i.d. sample, the pointwise law of large numbers holds for $\left(H(x,\xi)-t\right)^2$ and $H(x,\xi)$. Because $\left(H(x,\xi)-t\right)^2$ and $H(x,\xi)$ are convex in $(x,t)$, by Theorem 7.55 of Shapiro et al. (2014), the uniform law of large numbers in $X\times T$ also holds for them. Then, the assumptions in Theorem 5.3 of Shapiro et al. (2014) are satisfied. Following the analysis on Page 181 of Shapiro et al. (2014), condition (a) in Theorem 5.5 of Shapiro et al. (2014) holds. Because Assumption \ref{assu:Slater} is satisfied, following the analysis on Page 182 of Shapiro et al. (2014), condition (b) in Theorem 5.5 of Shapiro et al. (2014) also holds. It follows from Theorem 5.5 of Shapiro et al. (2014) that $v_n \to v^*$ w.p.1, and $(x_n,t_n)\to (x^*,t^*)$ w.p.1.

Note that the set of optimal solutions $S=\left\{(x^*,t^*)\right\}$ is a singleton, and $t^*=\E\left[H(x^*,\xi)\right]$. Suppose that Assumptions \ref{assu:Slater}, \ref{assu:lip} and \ref{assu:lip:variance} are satisfied. It follows from Theorem 5.11 of Shapiro et al. (2014) that the remain asymptotic results hold. 
This concludes the proof of the theorem.
\hfill$\Box$

An important observation from the theorem is that the variability of the sample mean-variance model depends on the first four moments of the random returns $\xi$. When the distribution of $\xi$ is sufficiently light-tailed, Assumption \ref{assu:lip:variance} typically holds. It can ensure the existence of the asymptotic variance $\sigma^2$ in (\ref{CLTVvmin}). However, when the distribution tail of $\xi$ becomes heavier, Assumption \ref{assu:lip:variance} may be violated and $\sigma^2$ may blow up. In the numerical experiments, we will conduct SAA for the $t$-distribution.

The optimal solution of Problem (\ref{MVportfolio}) depends on $X$. In practice, the decision makers may impose different constraints on the portfolio weights. For some $X$, it is possible to derive analytical solution for Problem (\ref{MVportfolio}). We recall some classical results of the mean-variance model. Consider $X=\left\{x: {\bf 1}^{\rm T} x= 1\right\}$ where ${\bf 1}$ is the all one vector. Note that this set is not compact. But we can impose very loose bounds on $x$ to make the set compact. We consider solving Problem (\ref{MVportfoliodet}) based on the Lagrangian approach. Let $\lambda$ and $\nu$ denote the Lagrangian multipliers associated with the constraints $\mu^{\rm T} x \ge r$ and ${\bf 1}^{\rm T} x= 1$, respectively. Let $\lambda^*$ and $\nu^*$ denote the optimal multipliers corresponding to the optimal solution $x^*$. Consider the following two cases. The results can be found in Burke (2020).   

Case 1: $\mu^{\rm T} x^* = r$. In this case, we exclude the setting where $\mu$ and ${\bf 1}$ are linear dependent. Consider that $\mu$ and ${\bf 1}$ are linear independent. Then it can verified that 
\begin{equation}
\delta:=\left(\mu^{\rm T}\Sigma^{-1} \mu\right) \left({\bf 1}^{\rm T}\Sigma^{-1} {\bf 1}\right)-\left(\mu^{\rm T}\Sigma^{-1} {\bf 1}\right)^2>0. \nonumber
\end{equation}
The optimal solution is 
\begin{equation*}
x^*=\left(1-\alpha\right)\frac{\Sigma^{-1} {\bf 1}}{{\bf 1}^{\rm T}\Sigma^{-1} {\bf 1}}+ \alpha \frac{\Sigma^{-1} \mu}{{\bf 1}^{\rm T}\Sigma^{-1} \mu},
\end{equation*}
where 
\begin{equation*}
\alpha=\delta^{-1}\left[ r\left(\mu^{\rm T}\Sigma^{-1} {\bf 1}\right) \left({\bf 1}^{\rm T}\Sigma^{-1} {\bf 1}\right)-\left(\mu^{\rm T}\Sigma^{-1} {\bf 1}\right)^2\right],
\end{equation*}
and the corresponding multipliers are 
\begin{equation*}
\lambda^*=\delta^{-1} {\bf 1}^{\rm T}\Sigma^{-1}\left(r{\bf 1}-\mu\right), \quad \nu^*=- \delta^{-1} \mu^{\rm T}\Sigma^{-1}\left(r{\bf 1}-\mu\right).
\end{equation*}
For this case, we can obtain that the Lagrangian multiplier is unique. Thus, the asymptotic normality in Theorem \ref{thm:asyoptval} holds.  

Case 2: $\mu^{\rm T} x^* > r$. In this case, $\lambda^*=0$, $\nu^*=1/\left({\bf 1}^{\rm T}\Sigma^{-1} {\bf 1}\right)$, $x^*=\nu^*\Sigma^{-1} {\bf 1}$. Then the uniqueness of the Lagrangian multiplier is also guaranteed and the asymptotic normality holds. Especially, we have that the asymptotic variance in (\ref{CLTVvmin}) becomes 
$\sigma^2=\mbox{V}\left[\left(H(x^*,\xi)-\E\left[H(x^*,\xi)\right]\right)^2\right]$ 
due to $\lambda^*=0$.

\subsection{Setting Extensions}

We now discuss more about the asymptotic properties when the modeling settings are slightly different. We first consider an alternative estimator for the covariance matrix. In the sample approximation discussed above, we used the biased estimator $\hat\Sigma$. This estimator exactly matches the SAA formulation. In statistics, the unbiased estimator $\hat\Sigma_u$ provided above may sometimes be preferred. If we use the unbiased estimator $\hat\Sigma_u$ and obtain the following approximation 
\begin{eqnarray}\label{MVsampleportfolioub}
&\displaystyle\mathop{\mbox{minimize}}_{x\in X} & x^{\rm T}\hat\Sigma_u x\\
&\mbox{subject to} & \hat\mu^{\rm T} x \ge r, \nonumber
\end{eqnarray}
we can still derive the asymptotic properties for the optimal value of the approximation. To see this, note that the optimal value of Problem (\ref{MVsampleportfolioub}) is $\frac{n}{n-1}v_n$ where $v_n$ is the optimal value of Problem (\ref{MVsampleportfolio}). It can be verified that $\sqrt{n} \left(\frac{n}{n-1}v_n-v^*\right)$ has the same asymptotic distribution as $\sqrt{n} \left(v_n-v^*\right)$. 

Next we consider the regularization in the mean-variance model. Because there exist sample errors, the performance of the sample optimal portfolio may not be satisfactory. Many studies proposed to impose some regularization for the portfolio weights to enhance the out-of-sample performance for the mean-variance model. A first class of models with regularization takes the following form   
\begin{eqnarray}\label{MVportdetregularobj}
&\displaystyle\mathop{\mbox{minimize}}_{x\in X} & x^{\rm T}\Sigma x+\rho(x)\\
&\mbox{subject to} & \mu^{\rm T} x \ge r. \nonumber \label{portdetregularobjcons}
\end{eqnarray}
In Problem (\ref{MVportdetregularobj}), $\rho(x)$ is a convex function of $x$. It is often specified as the $L_1$ or $L_2$ norm of $x$, see, e.g., Brodie et al. (2009). A second type of regularization proposes to add constraints for the weights and consider the following problem   
\begin{eqnarray}\label{MVportdetregularcon}
&\displaystyle\mathop{\mbox{minimize}}_{x\in X} & x^{\rm T}\Sigma x\\
&\mbox{subject to} & \mu^{\rm T} x \ge r, \ \rho(x)\le \rho_0, \nonumber \label{portdetregularobjcons}
\end{eqnarray}
where $\rho(x)$ is also a convex function. For instance, DeMiguel et al. (2009a) imposed norm constraint on $x$ to build the portfolio optimization model. By using the stochastic optimization approach discussed above, we can show that the consistency and the asymptotic normality hold for the sample approximations of Problems (\ref{MVportdetregularobj}) and (\ref{MVportdetregularcon}). The asymptotic variance has the same expression as (\ref{CLTVvmin}). The only difference is the optimal solution $x^*$. This observation suggests that imposing regularization on the portfolio weights will not affect the convergence rate of the sample approximation, but it may affect the optimal solution and thus the asymptotic variance.

\section{Alternative Mean-Variance Models}\label{sec:MVA}

There are alternative expressions for the mean-variance model. We now consider a formulation that optimizes the weighted sum of the mean and the variance
\begin{eqnarray}\label{MVWeightport}
&\displaystyle\mathop{\mbox{minimize}}_{x\in X} & \mbox{V}\left[H(x, \xi)\right]-w\E\left[H(x, \xi)\right].
\end{eqnarray}
In Problem (\ref{MVWeightport}), $w\ge 0$ is a prespecified parameter that adjusts the weights of the return and the risk. When $w=0$, we obtain a minimum-variance problem, which is widely studied in the literature, see, e.g., Jagannathan and Ma (2003). For the linear portfolio, the problem can be rewritten as 
\begin{eqnarray}\label{MVWeightportdet}
&\displaystyle\mathop{\mbox{minimize}}_{x\in X} & x^{\rm T}\Sigma x-w\mu^{\rm T} x.
\end{eqnarray}
This model is relatively easier to analyze because it does not involve the constraint. By plugging the estimators in (\ref{equ:estimator}) into Problem (\ref{MVWeightportdet}), we can obtain the following sample approximation problem  
\begin{eqnarray}\label{MVWeightportsample}
&\displaystyle\mathop{\mbox{minimize}}_{x\in X} & x^{\rm T}\hat \Sigma x-w\hat\mu^{\rm T} x.
\end{eqnarray}
Let $v_n$ and $v^*$ denote the optimal values of Problems (\ref{MVWeightportsample}) and (\ref{MVWeightportdet}) respectively. We analyze the statistical properties of the sample approximation. Similarly as in the previous section, Problem (\ref{MVWeightportdet}) is equivalent to the following stochastic  optimization problem 
\begin{eqnarray}\label{MVWeightportExp}
&\displaystyle\mathop{\mbox{minimize}}_{x\in X, t\in\Re} & \E\left[\left(H(x, \xi)-t\right)^2-w H(x, \xi)\right].
\end{eqnarray}
Problem (\ref{MVWeightportsample}) is equivalent to the following SAA of Problem (\ref{MVWeightportExp}) 
\begin{eqnarray}\label{MVWeightportsaa}
&\displaystyle\mathop{\mbox{minimize}}_{x\in X, t\in\Re} & \frac{1}{n}\sum_{i=1}^n\left[\left(H(x, \xi_i)-t\right)^2-w H(x, \xi_i)\right].
\end{eqnarray}
Similarly, we specify a compact interval $T$ and use $t\in T$ to replace $t\in \Re$. Problem (\ref{MVWeightportExp}) has a unique optimal solution $(x^*,t^*)$. Let $\left(x_n, t_n\right)$ denote an optimal solution of Problem (\ref{MVWeightportsaa}). We make the following assumptions. 
\begin{assumption}\label{ass:finite}
For some $(x,t) \in X\times T$, $\E\left[\left[\left(H(x, \xi)-t\right)^2-w H(x, \xi)\right]^2\right]<+\infty$. 
\end{assumption}

\begin{assumption}\label{ass:Lipweight}
There exists a random function $\tilde K(\xi)$ with
$\E\left[\tilde K(\xi)^2\right]<\infty$ such that
\[\left|\left[\left(H(x_1,\xi)-t_1\right)^2-w H(x_1, \xi)\right]-\left[\left(H(x_2,\xi)-t_2\right)^2-w H(x_2, \xi)\right]\right|\le \tilde K(\xi)\
\|(x_1,t_1)-(x_2,t_2)\|, \] 
$\forall\ (x_1,t_1), (x_2,t_2)\in X\times T,\ a.e. \xi\in \Xi$.
\end{assumption}
Then we have the following proposition. 
\begin{proposition}
Suppose that Assumptions \ref{ass:finite} and \ref{ass:Lipweight} are satisfied. Then $v_n \to v^*$ w.p.1, $(x_n,t_n)\to (x^*,t^*)$ w.p.1, and 
\begin{equation*}
\sqrt{n} \left(v_n-v^*\right) \Rightarrow N(0,\sigma^2), 
\end{equation*}
where
\begin{equation}\label{CLTVwvmmin}
\sigma^2=\mbox{V}\left[\left(H(x^*,\xi)-\E\left[H(x^*,\xi)\right]\right)^2-w H(x^*,\xi)\right].
\end{equation}

\end{proposition}
\noindent\textbf{Proof.} The consistency follows from Theorem 5.4 of Shapiro et al. (2014). Note that Problem (\ref{MVWeightportExp}) has a unique optimal solution $(x^*,t^*)$ where $t^*=\E\left[H(x^*,\xi)\right]$. Then it follows from Shapiro et al. (2014) that the asymptotic normality holds. This concludes the proof of the proposition.
\hfill$\Box$

It is interesting to note that the variances in (\ref{CLTVwvmmin}) and (\ref{CLTVvmin}) share a similar structure. The difference lies in whether to use the prespecified weight or the optimal Lagrangian multiplier.     

We now consider a third mean-variance model which takes the following expression
\begin{eqnarray}\label{MVportdetVcon}
&\displaystyle\mathop{\mbox{minimize}}_{x\in X} & -\mu^{\rm T} x\\
&\mbox{subject to} &  x^{\rm T}\Sigma x \le \sigma^2, \nonumber \label{MVportfoliodetcons}
\end{eqnarray}
where $ \sigma^2$ is some prespecified risk threshold, and consider its sample approximation 
\begin{eqnarray}\label{MVportdetVconsa}
&\displaystyle\mathop{\mbox{minimize}}_{x\in X} & -\hat\mu^{\rm T} x\\
&\mbox{subject to} &  x^{\rm T}\hat\Sigma x \le \sigma^2.\nonumber \label{MVportfoliodetcons}
\end{eqnarray}
Problem (\ref{MVportdetVcon}) is equivalent to the following problem 
\begin{eqnarray}\label{MVportVconstrexp}
&\displaystyle\mathop{\mbox{minimize}}_{x\in X, t\in\Re} & \E\left[-H(x, \xi)\right]\\
&\mbox{subject to} & \E\left[\left(H(x, \xi)-t\right)^2\right]\le \sigma^2.\label{MVportVconstcon}
\end{eqnarray}
Problem (\ref{MVportdetVconsa}) is equivalent to the SAA of Problem (\ref{MVportVconstrexp})
\begin{eqnarray}\label{MVportVconstrsaa}
&\displaystyle\mathop{\mbox{minimize}}_{x\in X, t\in\Re} & \frac{1}{n}\sum_{i=1}^n\left[-H(x, \xi_i)\right]\\
&\mbox{subject to} & \frac{1}{n}\sum_{i=1}^n\left[\left(H(x, \xi_i)-t\right)^2\right]\le \sigma^2. \nonumber \label{MVportdetconst}
\end{eqnarray}
As in the analysis above, we specify a compact interval $T$ and use $t\in T$ to replace $t\in \Re$ in Problems (\ref{MVportVconstrexp}) and (\ref{MVportVconstrsaa}). We make the following assumptions. 
\begin{assumption}\label{assu:SlaterV}
There exists a point $x\in X$ such that $x^{\rm T}\Sigma x < \sigma^2$. 
\end{assumption}
This assumption implies that there exists $(x,t)\in X\times T$ such that $\E\left[\left(H(x, \xi)-t\right)^2\right]< \sigma^2$, i.e., the Slater condition holds for constraint (\ref{MVportVconstcon}). 
Different from the previous two models, the uniqueness of the optimal solution of Problem (\ref{MVportVconstrexp}) may not be guaranteed. Let $S$ denote the set of optimal solutions of Problem (\ref{MVportVconstrexp}), and $\Lambda$ denote the set of associated Lagrangian multipliers. Let $S_n$ denote the set of optimal solutions of Problem (\ref{MVportVconstrsaa}). Let $D\left(S_n,S\right)$ denote the deviation of $S_n$ to $S$. That is, $D\left(S_n,S\right)=\sup_{x\in S_n} d(x, S)$ where $d(x, S)=\inf_{y\in S} \|x-y\|$. Define 
\begin{equation}
Y(x,t,\lambda)= \lim_{n\to \infty} \sqrt{n}\left[\frac{1}{n}\sum_{i=1}^n\left[- H(x,\xi_i)+\lambda\left(H(x,\xi_i)-t\right)^2\right]-\E\left[- H(x,\xi)+\lambda \left(H(x,\xi)-t\right)^2\right]\right]. 
\end{equation}
By the CLT, $Y(x,t,\lambda)$ has a normal distribution $N(0,\sigma^2(x,t,\lambda))$ with 
\[\sigma^2(x,t,\lambda)=\mbox{V}\left[- H(x,\xi)+\lambda \left(H(x,\xi)-t\right)^2\right].\] 
By Theorem 5.11 of Shapiro et al. (2014), we can obtain the following result. 
\begin{proposition}\label{pro:MVVconstr}
Suppose that Assumptions \ref{assu:lip}, \ref{assu:lip:variance}, \ref{assu:SlaterV} are satisfied. Then $v_n \to v^*$ w.p.1, $D\left(S_n,S\right)\to 0$ w.p.1, and 
\begin{equation*}
\sqrt{n} \left(v_n-v^*\right) \Rightarrow \inf_{(x,t)\in S} \sup_{\lambda\in \Lambda}\ Y(x,t, \lambda).
\end{equation*}
If, moreover, $S=\left\{(x^*,t^*)\right\}$ and $\Lambda=\left\{\lambda^*\right\}$ are singletons, then  
\begin{equation*}
\sqrt{n} \left(v_n-v^*\right) \Rightarrow N(0,\sigma^2), 
\end{equation*}
where
\begin{equation}\label{equ:normv3}
\sigma^2=\mbox{V}\left[-H(x^*,\xi)+\lambda^*\left(H(x^*,\xi)-\E\left[H(x^*,\xi)\right]\right)^2\right].
\end{equation}
\end{proposition}

The proof of Proposition \ref{pro:MVVconstr} is similar to that of Theorem \ref{thm:asyoptval} and thus is omitted. The asymptotic variance in (\ref{equ:normv3}) also depends on the first four moments of random vector $\xi$. If the constraint (\ref{MVportVconstcon}) is not binding at $(x^*,t^*)$, i.e., $\E\left[\left(H(x^*, \xi)-t^*\right)^2\right]< \sigma^2$, then the optimal multiplier $\lambda^*=0$ by the complementary slackness (Boyd and Vandenberghe 2004). In this case, the asymptotic variance $\sigma^2=\mbox{V}\left[-H(x^*,\xi)\right]$ which only depends on the first two moments of the random returns. An intuitive interpretation of this result is that adding or removing the risk constraint does not affect the optimal solution of the portfolio optimization. The sample approximation in this case may tend to have a smaller variation.

\section{Distributionally Robust Mean-Variance Portfolio Optimization}\label{sec:DRO}

In the previous sections, we have discussed the statistical properties of the sample based mean-variance model. It can be seen that there may be significant variability for this data-driven approach. To address the uncertainty in portfolio optimization, the distributionally robust optimization (DRO) approach has been widely studied and adopted, see, e.g., Blanchet et al. (2022) and the references therein. In this paper, we study two classes of ambiguity sets modeled by the $\phi$-divergence and the Wasserstein distance respectively. We follow Blanchet et al. (2022) and consider the following DRO model
\begin{eqnarray}\label{DRO}
&\displaystyle\mathop{\mbox{minimize}}_{x\in X} &  \displaystyle\mathop{\mbox{maximize}}_{P\in \mathbb{P}} \ \mbox{V}_P\left[H(x, \xi)\right] \\
&\mbox{ subject to} & \displaystyle\mathop{\mbox{minimize}}_{P\in
\mathbb{P}} \ \E_P\left[H(x, \xi)\right]
\ge r, \nonumber
\end{eqnarray}
where $\E_P$ and $\mbox{V}_P$ denote that the expectation and the variance are taken with respect to $P$, and $\mathbb{P}$ is an ambiguity set. As emphasized in Blanchet et al. (2022), a main obstacle for addressing Problem (\ref{DRO}) is that the variance function is nonlinear in $P$. We highlight that this difficulty could be overcome by reformulating the problem in an expectation form, as outlined in the previous sections. Specifically, we reformulate Problem (\ref{DRO}) as   
\begin{eqnarray}\label{DROExp}
&\displaystyle\mathop{\mbox{minimize}}_{x\in X, t\in\Re} &  \displaystyle\mathop{\mbox{maximize}}_{P\in \mathbb{P}} \  \E_P\left[\left(H(x, \xi)-t\right)^2\right]\\
&\mbox{ subject to} & \displaystyle\mathop{\mbox{minimize}}_{P\in
\mathbb{P}} \ \E_P\left[H(x, \xi)\right]
\ge r.\label{DROExpconst}\nonumber
\end{eqnarray}
Based on the reformulation we can invoke the duality theory to transform the DRO model to a single layer optimization problem.

\subsection{$\phi$-Divergence}

We first consider an ambiguity set modeled by the $\phi$-divergence. The $\phi$-divergence is induced by a $\phi$-divergence function $\phi(t), t>0$. The function is convex in $t>0$, and satisfies $\phi(1)=0$, $0\phi(a/0):= a\lim_{t\to \infty} \phi(t)/t$, and $0\phi(0/0):=0$, see the definitions in Pardo (2006) and Ben-Tal et al. (2013). For two distributions $P$ and $P_0$, the $\phi$-divergence from $P$ to $P_0$ is defined as
\begin{equation*}\label{phiDEF}
D_{\phi}(P\|P_0) = \int_{\Xi} p_0(z)
\phi\left(\frac{p(z)}{p_0(z)}\right) \, {\rm
d}z=\E_{P_0}\left[\phi\left(\frac{p(\xi)}{p_0(\xi)}\right)\right], 
\end{equation*} 
where $p(z)$ and $p_0(z)$ are probability density (or mass) functions of $P$ and $P_0$ respectively, see the discussions in Hu and Hong (2022). The ambiguity set modeled by the $\phi$-divergence is defined as
\begin{equation}\label{ambsetphi} 
\mathbb{P}=\left\{P: D_{\phi}(P\|P_0)\le \delta \right\},  
\end{equation} 
where $P_0$ is a nominal distribution of $\xi$ and $\delta$ is a constant measuring the size of the ambiguity set. Note that to construct $\mathbb{P}$, we first need to specify $P_0$ based on the data or information available. We consider two representative $\phi$-divergences, namely the Kullback-Leibler (KL) divergence and the $\chi^2$-distance.   

Consider Problem (\ref{DROExp}) with the ambiguity set (\ref{ambsetphi}).  Suppose that $D_{\phi}$ is the KL divergence, for which $\phi(t)=t\log t$. Let $M_1(s)$ and $M_2(s)$ denote the moment generating functions of $\left(H(x, \xi)-t\right)^2$ and $-H(x, \xi)$ under $P_0$ respectively. Assume that for every $x\in X$ and $t\in \Re$, there exist $s_1$ and $s_2$ such that $M_1(s_1)<+\infty$ and $M_2(s_2)<+\infty$. Then by the strong duality (see, e.g., Hu and Hong (2012, 2022), Problem (\ref{DROExp}) is equivalent to the following problem
\begin{eqnarray}\label{DROKL}
&\displaystyle\mathop{\mbox{minimize}}_{x\in X, t\in\Re, \lambda_1,\lambda_2\in \Re, \alpha_1,\alpha_2\ge 0} & \E_{P_0}\left[\alpha_1
\exp\left\{\frac{\left(H(x, \xi)-t\right)^2+\lambda_1}{\alpha_1}-1\right\}\right]+\delta\alpha_1-\lambda_1 \\
&\mbox{ subject to} & \E_{P_0}\left[\alpha_2
\exp\left\{\frac{-H(x, \xi)+\lambda_2}{\alpha_2}-1\right\}\right]+\delta \alpha_2-\lambda_2  \le -r, \nonumber
\end{eqnarray}
which is also equivalent to 
\begin{eqnarray*}\label{DROKLlog}
	&\displaystyle\mathop{\mbox{minimize}}_{x\in X, t\in\Re, \alpha_1,\alpha_2\ge 0} & \alpha_1 \ln \E_{P_0}\left[
	\exp\left\{\frac{\left(H(x, \xi)-t\right)^2}{\alpha_1}\right\}\right]+\delta\alpha_1 \\
	&\mbox{ subject to} & \alpha_2 \ln \E_{P_0}\left[
	\exp\left\{\frac{-H(x, \xi)}{\alpha_2}\right\}\right]+\delta \alpha_2  \le -r. \nonumber
\end{eqnarray*}

Suppose that $D_{\phi}$ is the $\chi^2$-distance, for which $\phi(t)=(t-1)^2$. Assume that for every $x\in X$ and $t\in \Re$, $\E_{P_0}\left[\left(H(x, \xi)-t\right)^4\right]<+\infty$ and $\E_{P_0}\left[\left(H(x, \xi)\right)^2\right]<+\infty$. Then Problem (\ref{DROExp}) is equivalent to the following problem % (see, e.g., Hu and Hong (2022))
\begin{eqnarray}\label{DROChi}
&\displaystyle\mathop{\mbox{minimize}}_{x\in X, t\in\Re, \lambda_1,\lambda_2\in \Re, \alpha_1,\alpha_2\ge 0} & \frac{1}{4\alpha_1}\E_{P_0}\left[\left(\left[\left(H(x, \xi)-t\right)^2+\lambda_1+2\alpha_1\right]^+\right)^2\right]+(\delta-1)\alpha_1
-\lambda_1 \\
&\mbox{ subject to} & \frac{1}{4\alpha_2}\E_{P_0}\left[\left(\left[-H(x, \xi)+\lambda_2+2\alpha_2\right]^+\right)^2\right]+(\delta-1)\alpha_2
-\lambda_2 \le -r. \nonumber
\end{eqnarray}
As standard convex stochastic optimization problems, the reformulations (\ref{DROKL}) and (\ref{DROChi}) can be solved by the conventional sampling based stochastic optimization techniques. Note that directly addressing the worst-case variance with (\ref{ambsetphi}) might involve complicated analysis. After making the transformation, the reformulation becomes relatively easy. The framework can also be adopted naturally to build reformulations for the DRO models of the two other mean-variance variants considered in Section \ref{sec:MVA}.

\subsection{Wasserstein Distance}

We now consider an ambiguity set modeled by the Wasserstein distance. Problem (\ref{DRO}) with this ambiguity set was studied in Blanchet et al. (2022). We follow Blanchet et al. (2022) to introduce the Wasserstein distance and the ambiguity set. Let $\mathcal{P}\left(\Re^k\times \Re^k \right)$ be the set of distributions supported on $\Re^k\times \Re^k$. An element $\pi\in \mathcal{P}\left(\Re^k\times \Re^k \right)$ is a joint distribution of the random vector $\left(U,V\right)$ with $U,V\in \Re^k$. Let $\pi_U$ and $\pi_V$ denote the marginal distributions of $U$ and $V$ under $\pi$. Let $c: \Re^k\times \Re^k \to [0,+\infty]$ be a cost function which is lower semicontinuous satisfying $c(u,u)=0$ for any $u\in \Re^k$. More specifically, consider 
\begin{equation*}
c\left(u,v\right)=\left\|u-v\right\|_q^2, 
\end{equation*}
where $\|\cdot\|_q$ denotes the $L_q$ norm and $q\ge 1$ is a given constant. The Wasserstein distance (of order two) between two distributions $P$ and $Q$ is defined as 
\begin{equation*}\label{Wassdis}
D_c\left(P, Q\right)=\inf \left\{\E_\pi \left[c(U,V)\right]: \pi\in \mathcal{P}\left(\Re^k\times \Re^k \right), \pi_U=P, \pi_V=Q\right\}. 
\end{equation*}
Suppose that $\left\{\xi_1,\xi_2,\cdots,\xi_n\right\}$ is a set of i.i.d. sample of $\xi$. Let $P_n$ denote the corresponding empirical distribution. Define the ambiguity set 
\begin{equation}\label{ambsetW}
\mathbb{P}=\left\{P: D_c\left(P, P_n\right) \le \delta\right\}. 
\end{equation}
Consider Problem (\ref{DRO}) with the ambiguity set (\ref{ambsetW}). Let $\|\cdot\|_p$ denote the dual norm of $\|\cdot\|_q$ where $\frac{1}{p}+\frac{1}{q}=1$. The following proposition is built by Blanchet et al. (2022). 
\begin{proposition}[Blanchet et al. (2022)]\label{pro:DRO}
Problem (\ref{DRO}) with the ambiguity set (\ref{ambsetW}) is equivalent to the following problem: 
\begin{eqnarray}\label{MVsampleportfolioDROWRef}
&\displaystyle\mathop{\mbox{minimize}}_{x\in X} & \sqrt{x^{\rm T}\hat\Sigma x} + \sqrt{\delta}\|x\|_p \\
&\mbox{subject to} & \hat\mu^{\rm T} x \ge r+ \sqrt{\delta}\|x\|_p. \nonumber
\end{eqnarray}
\end{proposition}

In what follows, we consider a DRO model for Problem  (\ref{MVWeightport}) that optimizes a weighted sum of mean and variance. 
\begin{eqnarray}\label{MVWeightportDROW}
&\displaystyle\mathop{\mbox{minimize}}_{x\in X} \
\displaystyle\mathop{\mbox{maximize}}_{P\in \mathbb{P}}  & \mbox{V}_P\left[H(x, \xi)\right]-w\E_P\left[H(x, \xi)\right].
\end{eqnarray}
When $w=0$, the objective of Problem (\ref{MVWeightportDROW}) becomes minimizing the worst-case variance. We build the following proposition that reformulates Problem (\ref{MVWeightportDROW}) to a single layer optimization problem. The proof of this proposition follows the main steps of that of Proposition A.1 in Blanchet et al. (2022) and uses techniques presented in the proof of Proposition 2 in Blanchet et al. (2019). We also illustrate that the proof of this proposition can be directly used to prove  Proposition \ref{pro:DRO}. 
\begin{proposition}
Problem (\ref{MVWeightportDROW}) with the ambiguity set (\ref{ambsetW}) is equivalent to the following problem: 
\begin{equation*}\label{}
\displaystyle\mathop{\mbox{minimize}}_{x\in X, t\in\Re} \ \frac{1}{n}\sum_{i=1}^n\left[\left(H(x, \xi_i)-t\right)^2-w H(x, \xi_i)\right]+\sqrt{\delta}\sqrt{\frac{1}{n}\sum_{i=1}^n\left(2\left(H(x, \xi_i)-t\right)-w\right)^2}\|x\|_p+\delta\|x\|_p^2.
\end{equation*}
\end{proposition}
\noindent\textbf{Proof.} 
Firstly, Problem (\ref{MVWeightportDROW}) is equivalent to 
\begin{eqnarray}\label{DROref0}
&\displaystyle\mathop{\mbox{minimize}}_{x\in X, t\in\Re} &  \displaystyle\mathop{\mbox{maximize}}_{P\in \mathbb{P}} \  \E_P\left[\left(H(x, \xi)-t\right)^2-w H(x, \xi) \right].\label{}
\end{eqnarray}
Let $v(x,t)$ denote the optimal value of the inner maximization problem in Problem (\ref{DROref0}). We derive a formulation for $v(x,t)$. By the strong duality (Proposition 1 of Blanchet et al. (2019)), we have 
\begin{equation*}
v(x,t)= \inf_{\gamma\ge 0} \left[\delta \gamma  + \frac{1}{n}\sum_{i=1}^n \Phi_{\gamma}\left(\xi_i; x,t\right)\right],
\end{equation*} 
where 
\begin{eqnarray*}
\Phi_{\gamma}\left(\xi_i; x,t\right)&=&\sup_{u}\left[\left(H(x, u)-t\right)^2-wH(x,u)-\gamma c(u,\xi_i)\right]\\
&=&\sup_{u}\left[\left(x^{\rm T} u-t\right)^2-wx^{\rm T} u-\gamma \left\|u-\xi_i\right\|_q^2 \right].\nonumber
\end{eqnarray*}

Let $\Delta=u-\xi_i$. Then $u=\Delta+\xi_i$ and 
 \begin{eqnarray*}
\Phi_{\gamma}\left(\xi_i;x,t\right)&=&\sup_{\Delta}\left[\left(x^{\rm T} \Delta+x^{\rm T} \xi_i-t\right)^2-wx^{\rm T} \Delta-wx^{\rm T} \xi_i-\gamma \left\|\Delta \right\|_q^2 \right]\\
&=&\sup_{\Delta}\left[\left(x^{\rm T} \xi_i-t\right)^2+2\left(x^{\rm T} \xi_i-t\right)x^{\rm T} \Delta+\left(x^{\rm T} \Delta\right)^2 -wx^{\rm T} \Delta-wx^{\rm T} \xi_i-\gamma  \left\|\Delta \right\|_q^2\right]\\
&=&\sup_{\Delta}\left[\left(x^{\rm T} \xi_i-t\right)^2-wx^{\rm T} \xi_i+\left[2\left(x^{\rm T} \xi_i-t\right)-w\right]x^{\rm T} \Delta+\left(x^{\rm T} \Delta\right)^2 -\gamma  \left\|\Delta \right\|_q^2\right]\\
&=&\sup_{\Delta}\left[\left(x^{\rm T} \xi_i-t\right)^2-wx^{\rm T} \xi_i+\left|2\left(x^{\rm T} \xi_i-t\right)-w\right| \left\|x\right\|_p\left\|\Delta\right\|_q+\left(\left\|x\right\|_p\left\|\Delta\right\|_q\right)^2 -\gamma  \left\|\Delta \right\|_q^2\right]\\
&=&\sup_{\Delta}\left[\left(x^{\rm T} \xi_i-t\right)^2-wx^{\rm T} \xi_i+\left|2\left(x^{\rm T} \xi_i-t\right)-w\right| \left\|x\right\|_p\left\|\Delta\right\|_q+\left(\left\|x\right\|_p^2-\gamma\right) \left\|\Delta \right\|_q^2\right]\\
&=&\left\{\begin{array}{cc} \left(x^{\rm T} \xi_i-t\right)^2-wx^{\rm T} \xi_i + \frac{\left(2\left(x^{\rm T} \xi_i-t\right)-w\right)^2\left\|x\right\|_p^2}{4\left(\gamma -\left\|x\right\|_p^2\right)} &
	\left\|x\right\|_p^2-\gamma< 0 \\[6pt]  
+\infty, & \left\|x\right\|_p^2-\gamma\ge 0 
\end{array}\right. .\nonumber
\end{eqnarray*}

Denote $a=\frac{1}{n}\sum_{i=1}^n \left[\left(x^{\rm T} \xi_i-t\right)^2-wx^{\rm T} \xi_i\right]$ and $b=\frac{1}{n}\sum_{i=1}^n \left(2\left(x^{\rm T} \xi_i-t\right)-w\right)^2$. It follows that
\begin{eqnarray*}
v(x,t)&=&\inf_{\gamma> \left\|x\right\|_p^2}  \left\{ a+\delta\gamma +\frac{b\left\|x\right\|_p^2}{4\left(\gamma -\left\|x\right\|_p^2\right)}\right\}\\
&=&\inf_{\gamma> \left\|x\right\|_p^2} \left\{ a+\delta\left(\gamma-\left\|x\right\|_p^2\right) +\frac{b\left\|x\right\|_p^2}{4\left(\gamma -\left\|x\right\|_p^2\right)}+\delta \left\|x\right\|_p^2\right\}\\
&=& a+\sqrt{\delta b} \left\|x\right\|_p +\delta \left\|x\right\|_p^2,\nonumber
\end{eqnarray*}
where the last equality holds since the infimum is attained when $\gamma=\sqrt{\frac{b\left\|x\right\|_p^2}{4\delta}}+\left\|x\right\|_p^2$. This concludes the proof of the proposition. \hfill$\Box$

We note that when $w=0$, the optimal value function becomes 
\begin{eqnarray*}
v(x,t) &=&\frac{1}{n}\sum_{i=1}^n\left(H(x, \xi_i)-t\right)^2+\sqrt{\delta}\sqrt{\frac{1}{n}\sum_{i=1}^n\left(2\left(H(x, \xi_i)-t\right)\right)^2}\|x\|_p+\delta\|x\|_p^2\\
&=&\left(\sqrt{\frac{1}{n}\sum_{i=1}^n\left(H(x, \xi_i)-t\right)^2}+\sqrt{\delta} \|x\|_p\right)^2.\nonumber
\end{eqnarray*}
According to the analysis in Section \ref{sec:MV}, we have 
\begin{equation*}
\inf_{t\in \Re} \ v(x,t)= \left(\sqrt{x^{\rm T}\hat\Sigma x} + \sqrt{\delta}\|x\|_p\right)^2, 
\end{equation*}
where $\hat\Sigma$ is defined in (\ref{equ:estimator}). The same duality argument can be applied for the constraint in Problem (\ref{DRO}) (see Proposition A.1 of Blanchet et al. (2022)), yielding the equivalence of the constraints in Problem (\ref{DRO}) and Problem (\ref{MVsampleportfolioDROWRef}). This shows Proposition \ref{pro:DRO}. As pointed out by Blanchet et al. (2022), Proposition  \ref{pro:DRO} shows a clear relation between robustification and regularization in the mean-variance model. However, we observe that such relation seems to be more subtle when considering the weighted sum of mean and variance as the objective. 

In this section, we have illustrated that by writing the variance of the portfolio as the optimal objective function of an expectation, we can substantially simplify the analysis of the DRO of the mean-variance model with both the $\phi$-divergence and the Wasserstein distance.

\section{Numerical Illustrations}\label{sec:num}

In this section, we conduct some numerical experiments to justify the theoretical results built in the paper. As there is an extensive empirical study in Blanchet et al. (2022), we do not conduct experiments for the DRO. For illustrative purpose, we only solve the first mean-variance model. The experiments for the other models can be conducted accordingly. We consider Problem (\ref{MVportfolio}) with the set $X=\left\{x: {\bf 1}^{\rm T} x= 1\right\}$ and the number of assets $k=50$. We first assume that $\xi$ follows a multivariate normal distribution $N(\mu, \Sigma)$. For the parameters of $N(\mu, \Sigma)$, we adopt a configuration in Hong et al. (2014). The elements of $\mu=\left(\mu_1,\cdots,\mu_k\right)^{\rm T} $ evenly spread
between 0.04 and 0.50 and increase with the subscript. The standard deviation $\mbox{std}\left[\xi_i\right]$ is equal to $\mu_i+0.05$ for $i=1,\cdots,k$ and the correlation between any two elements $\xi^i$ and $\xi^j$ of $\xi$ is 0.35 where $i \not=j$. The true optimal solution and optimal value can be derived analytically based on the discussion in Section \ref{sec:MV}. They are used as the benchmark in the experiments. The sample approximation Problem (\ref{MVsampleportfolio}) is solved by the CVX package (Grant and Boyd 2020). 

In Table \ref{wscnum1}, we report the optimal value of Problem (\ref{MVportfolio}) and five typical replications of the sample optimal value for different sample sizes. We can see that the sample optimal value gradually tends to the true value as $n$ increases. For small sample sizes, the sample errors are relatively large. Especially, the results exhibit a negative bias to the true value. This is consistent to the fact that the SAA minimization yields negative bias.      

\begin{table}[htbp]
\centering\small \caption{\label{tab:PO}Mean-Variance Optimization with Normal Distribution}{
\begin{tabular}{ccccc}%\\
\hline  %& $r=0.1$ & $r=0.2$ & $r=0.3$ & $r=0.4$  \\
        & $n=10^2$ & $n=10^3$ & $n=10^4$ & $n=10^5$  \\
\hline  True Opt & 0.0328 & 0.0328  & 0.0328 & 0.0328\\
        1 & 0.0112 & 0.0279  & 0.0326  & 0.0329\\
        2 & 0.0150  & 0.0282  & 0.0323  & 0.0329\\
        3 & 0.0116 & 0.0325  & 0.0334  & 0.0327\\
        4 & 0.0105 & 0.0313  & 0.0325  & 0.0330\\
        5 & 0.0121 & 0.0302  & 0.0342  & 0.0328\\
\hline
\end{tabular}}\label{wscnum1}
\end{table}

We further test the convergence speed of the sample optimal value. Let $S_n=v_n-v^*$ and $S=\sqrt{n}(v_n-v^*)$. We simulate $S_n$ and $S$ for $1000$ times for
different $n$,  and plot the densities of $S_n$ and $S$ in Figure \ref{Fig:test1estimator}. The left panel of the figure shows the density and trend of $S_n$. When $n$ becomes larger, $S_n$ gradually concentrates around $0$. The right panel of Figure \ref{Fig:test1estimator} shows the density of $S$. The density appears to admit an asymptotic normality behavior when $n$ is sufficiently large. These observations support the result built in Theorem
\ref{thm:asyoptval}.

\begin{figure}[htp]
\centering \resizebox{8cm}{5cm}{
\includegraphics[scale=.35]{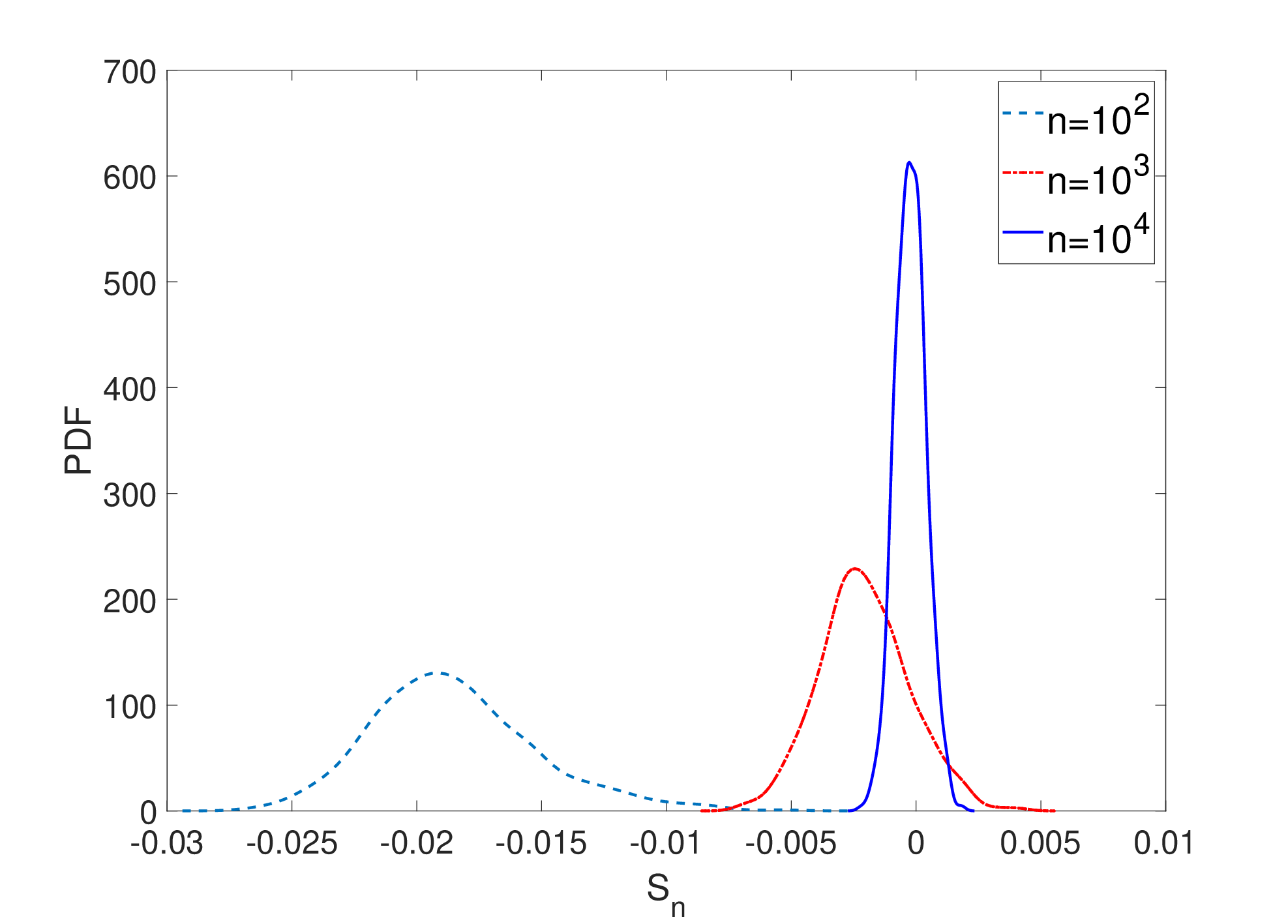}}
\hspace{-0.2cm} \resizebox{8cm}{5cm}{\includegraphics[scale=.35]{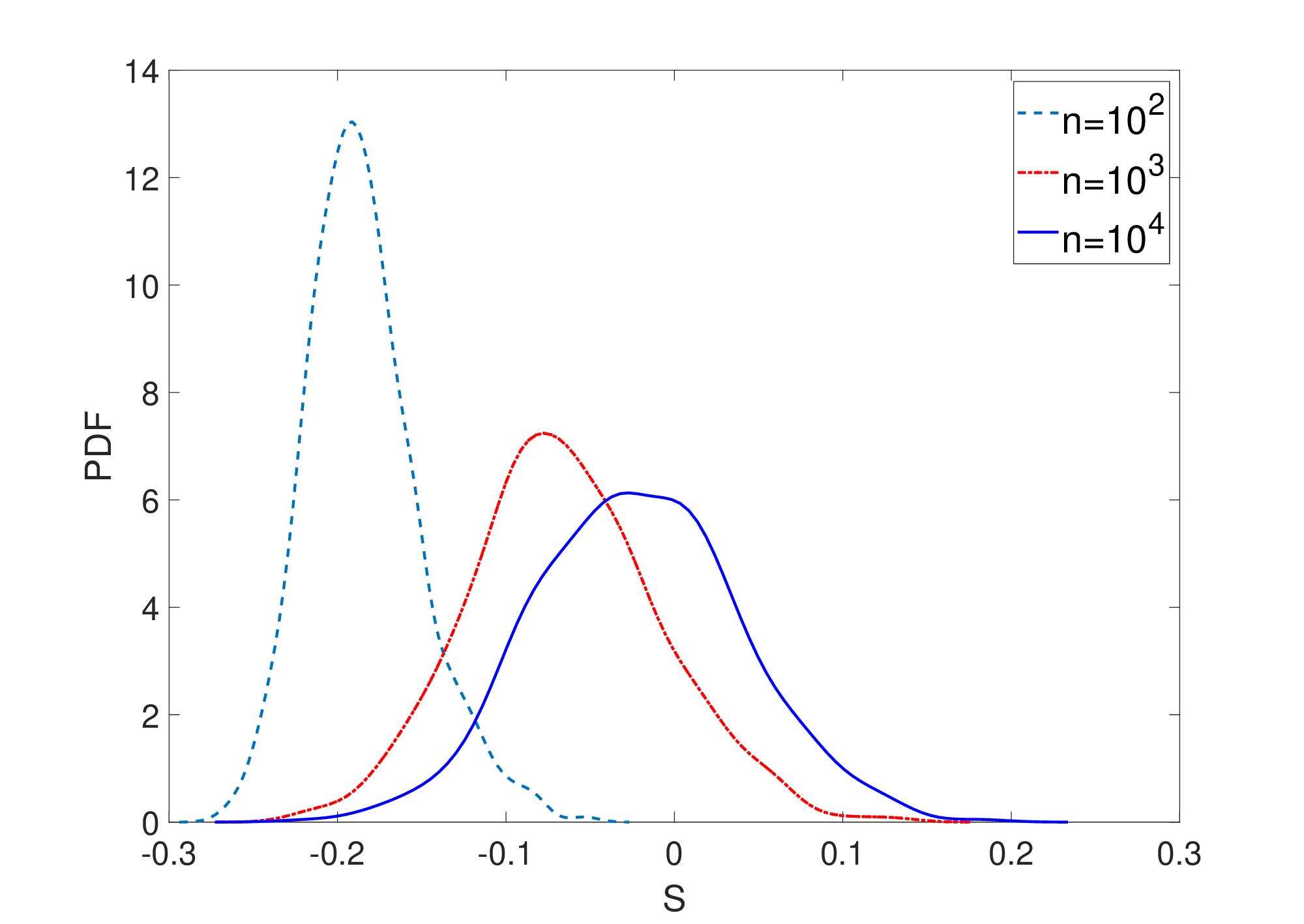}}
\caption{Performance of Sample Approximation with Normal Distribution}\label{Fig:test1estimator}
\end{figure}

We also test the effects of the regularization. Consider Problem (\ref{MVportdetregularcon}) with the set $X=\left\{x: x\ge 0, {\bf 1}^{\rm T} x\le 1\right\}$, and set $\rho(x)=\|x\|_1$ where $\|\cdot\|_1$ denotes the $L_1$ norm. The minimum of $\|x\|_1$ over all feasible $x$ is $0.5094$. We consider $\rho_0=0.55, 0.80$, and solve  Problem (\ref{MVportdetregularcon}) to obtain $v^*$ for the two settings. We simulate $S$ defined above for $1000$ times with $n=10^4$ and plot the densities of $S$ in Figure \ref{Fig:test1estimator_Reg}. From the figure, we observe that with regularization, the sample optimal value still admit an asymptotic
normality behavior.        

\begin{figure}[htp]
\centering \resizebox{8cm}{5cm}{
\includegraphics[scale=.35]{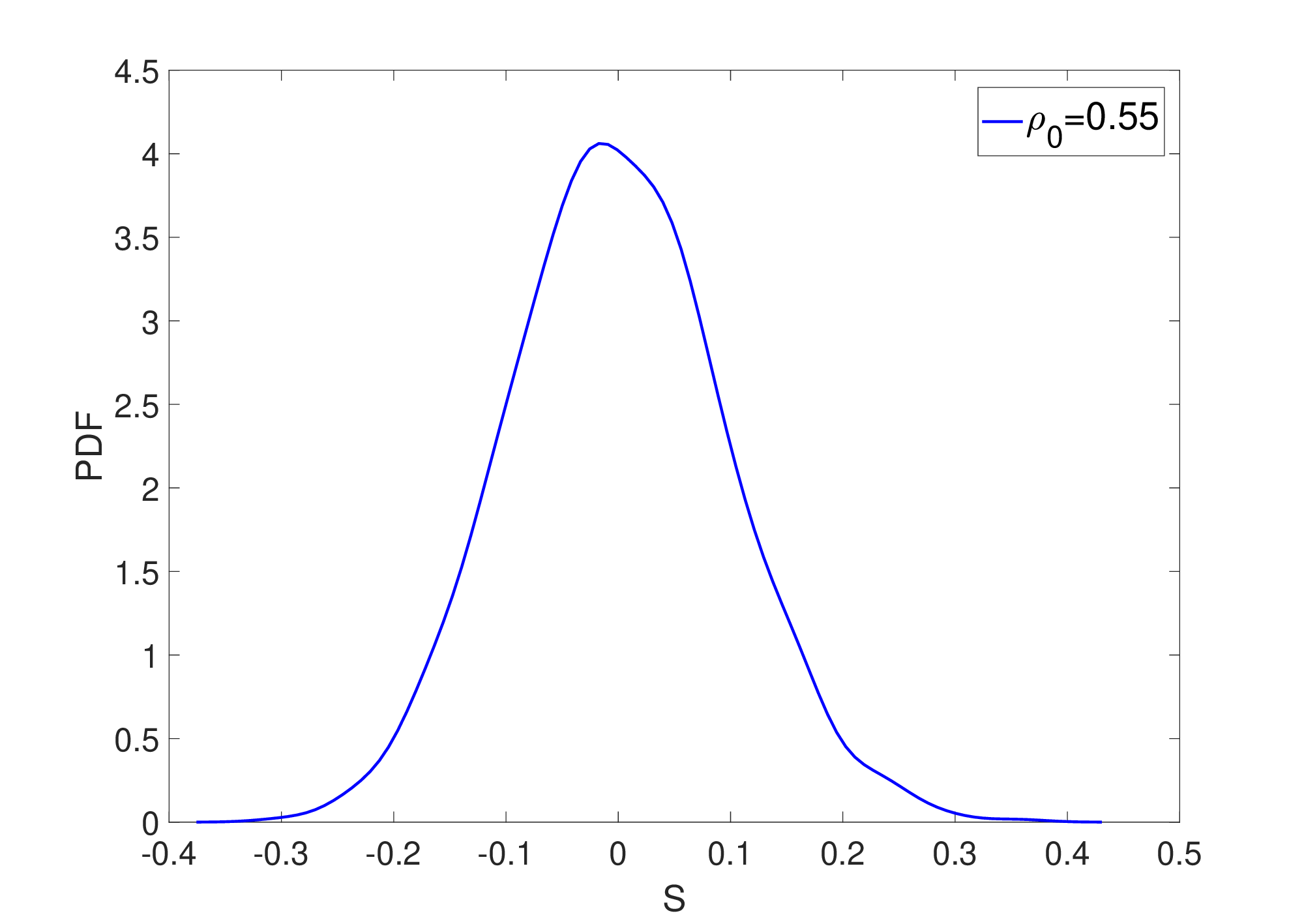}}
\hspace{-0.2cm} \resizebox{8cm}{5cm}{\includegraphics[scale=.35]{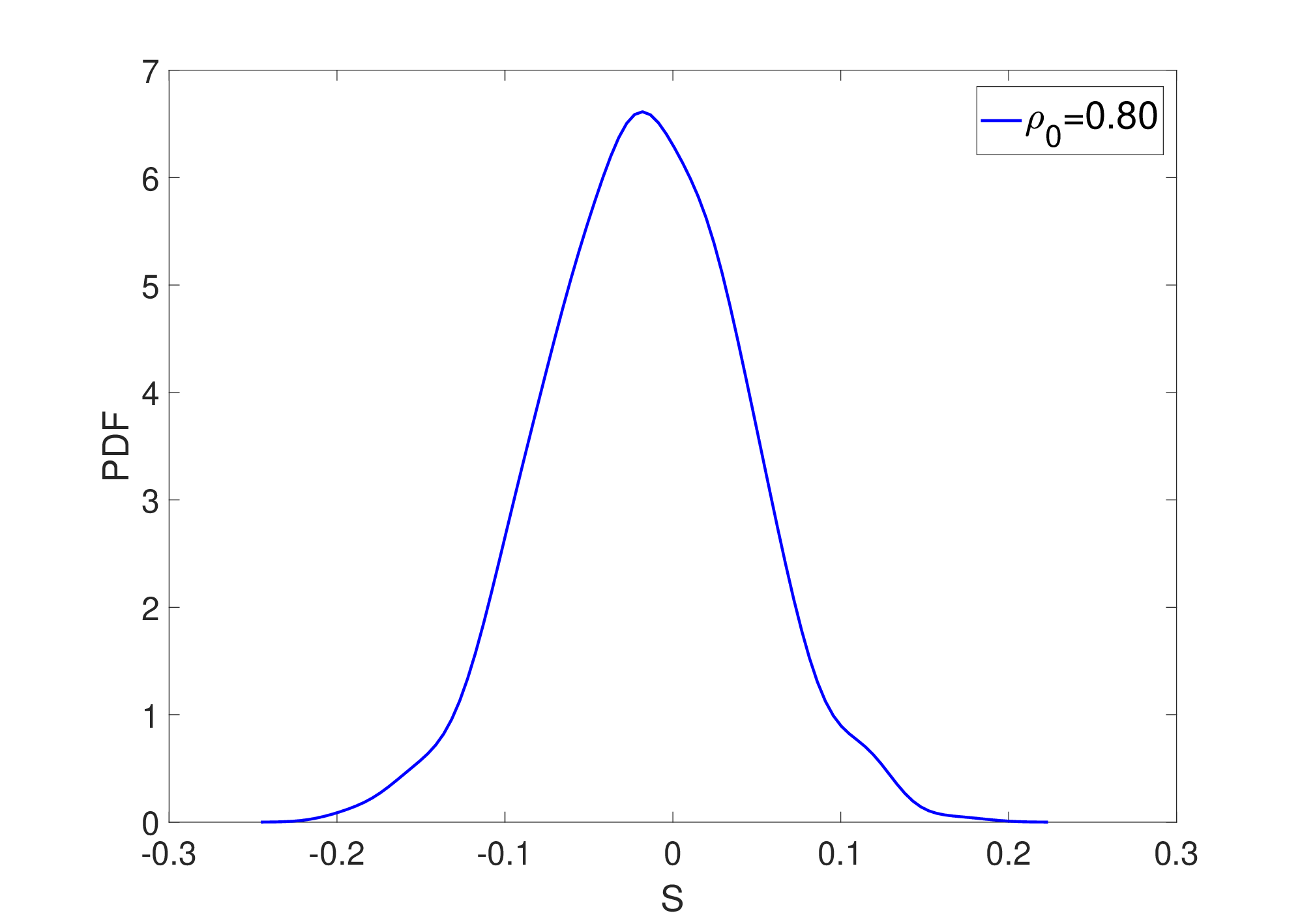}}
\caption{Performance of Sample Approximation with Normal Distribution with Regularization}\label{Fig:test1estimator_Reg}
\end{figure}

Now we consider a distribution with a heavier tail. Let $R=\mbox{Chol}(\Sigma)$ denote the Cholesky decomposition of $\Sigma$. Then $\Sigma=R^{\rm T}R$. We assume that 
$\xi=R^{\rm T} \left(\sqrt{\frac{\nu-2}{\nu}}\eta \right)+\mu$
where $\eta$ is a $k$-dimensional random vector with each element following a $t$ distribution with degree of freedom $\nu$ and all elements being independent. We can compute that the mean vector and the covariance matrix of $\xi$ are still $\mu$ and $\Sigma$ provided that $\nu>2$. Therefore, the true optimization problem remains the same. Recall that for a $t$ distribution with degree of freedom $\nu$, if $m>\nu$, the $m$-th moment of the $t$ distribution does not exist. In the experiment, we consider $\nu=5$ and $\nu=3$. For $\nu=5$, the fourth moment of $\eta$ exists. However, when $\nu=3$, the fourth moment of $\eta$ does not exist and in this setting the asymptotic variance in (\ref{CLTVvmin}) may not exist. We conduct experiments and examine it. The other settings are the same as above.  

Table \ref{wscnum2} summarizes five typical simulation replications for the sample optimal values for the two degrees of freedom. The results essentially shows a consistency behavior. Comparing the two parts of the table, we observe that the results for $\nu=3$ appear to have relatively larger variability than that for $\nu=5$ when $n=10^5$.        
\begin{table}[htbp]
\centering\small \caption{\label{tab:PO}Mean-Variance Optimization with $t$ Distribution}{
\begin{tabular}{ccccc||ccccc}%\\
\hline  
     $\nu=5$   & $n=10^2$ & $n=10^3$ & $n=10^4$ & $n=10^5$    &  $\nu=3$    & $n=10^2$ & $n=10^3$ & $n=10^4$ & $n=10^5$   \\
\hline  True Opt & 0.0328 & 0.0328  & 0.0328 & 0.0328 & True Opt & 0.0328 & 0.0328  & 0.0328 & 0.0328\\
        1 & 0.0148 & 0.0336  & 0.0329  & 0.0326 & 1 & 0.0103 & 0.0297  & 0.0318  & 0.0319\\
        2 & 0.0125 & 0.0299  & 0.0326  & 0.0329 & 2 & 0.0107 & 0.0348  & 0.0322  & 0.0317 \\
        3 & 0.0151 & 0.0338  & 0.0315  & 0.0332 &  3 & 0.0136 & 0.0284  & 0.0321  & 0.0324 \\
        4 & 0.0132 & 0.0299  & 0.0328  & 0.0327 & 4 & 0.0098 & 0.0279  & 0.0324  & 0.0322 \\
        5 & 0.0090 & 0.0289  & 0.0328  & 0.0325 & 5 & 0.0199 & 0.0298  & 0.0319  & 0.0317\\
\hline
\end{tabular}}\label{wscnum2}
\end{table}

Finally, we test the convergence speed of the sample approximation. We set $n=10^4$ and simulate $S$ for $1000$ times. Based
on the simulation we plot the density of $S$,
which are shown in Figure \ref{Fig:test2estimator} (left panel for $\nu=5$ and right panel for $\nu=3$). From the figure, we can see that when $\nu=5$, $S$ admits an asymptotic normality behavior. However, when $\nu=3$, the density exhibits non-symmetry and has a long right tail. It appears to have quite different behavior compared to the normal distribution. We will further explore the properties in the future study. 

\begin{figure}[h]
\centering \resizebox{8cm}{5cm}{
\includegraphics[scale=.35]{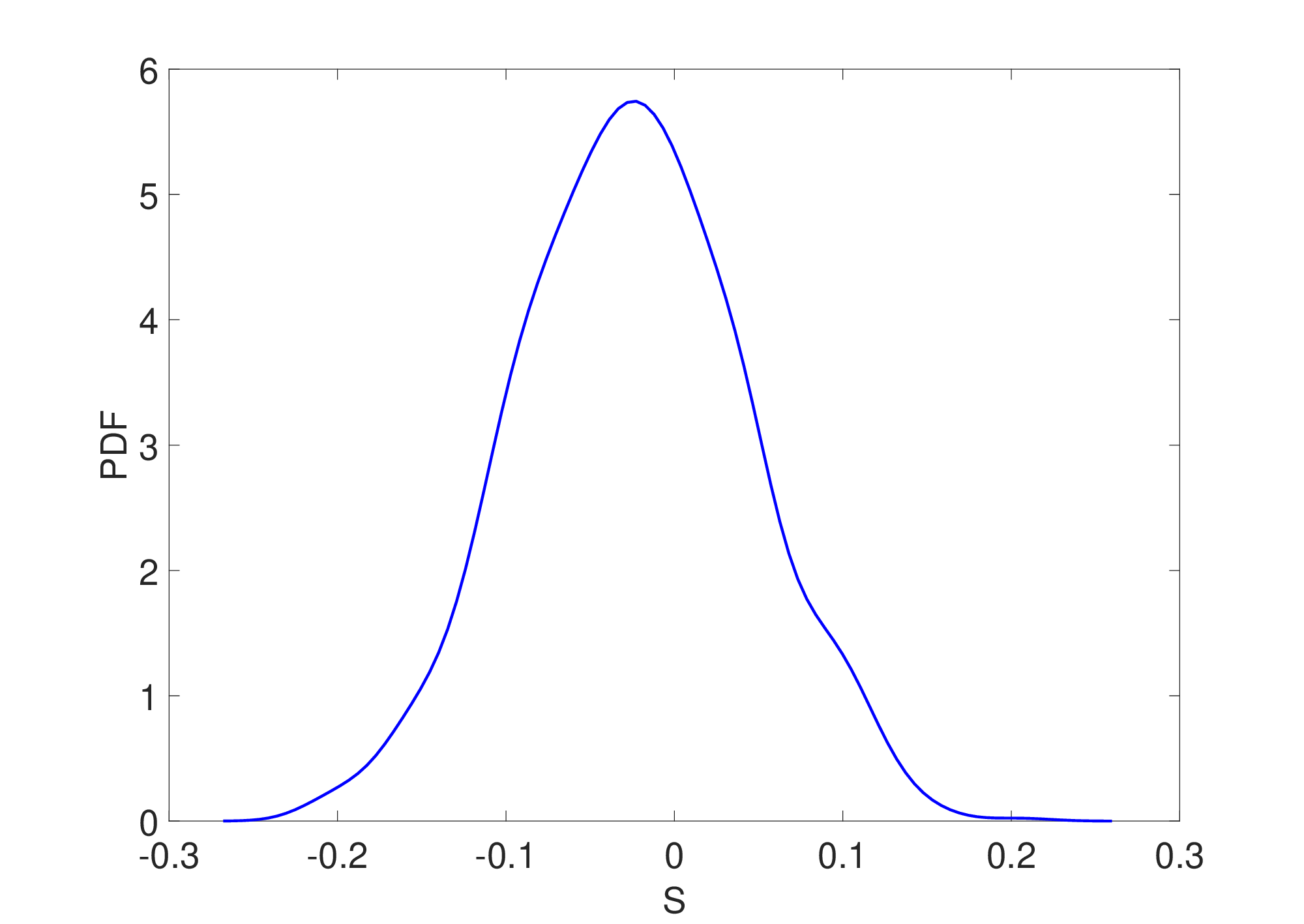}}
\hspace{-0.2cm} \resizebox{8cm}{5cm}{\includegraphics[scale=.35]{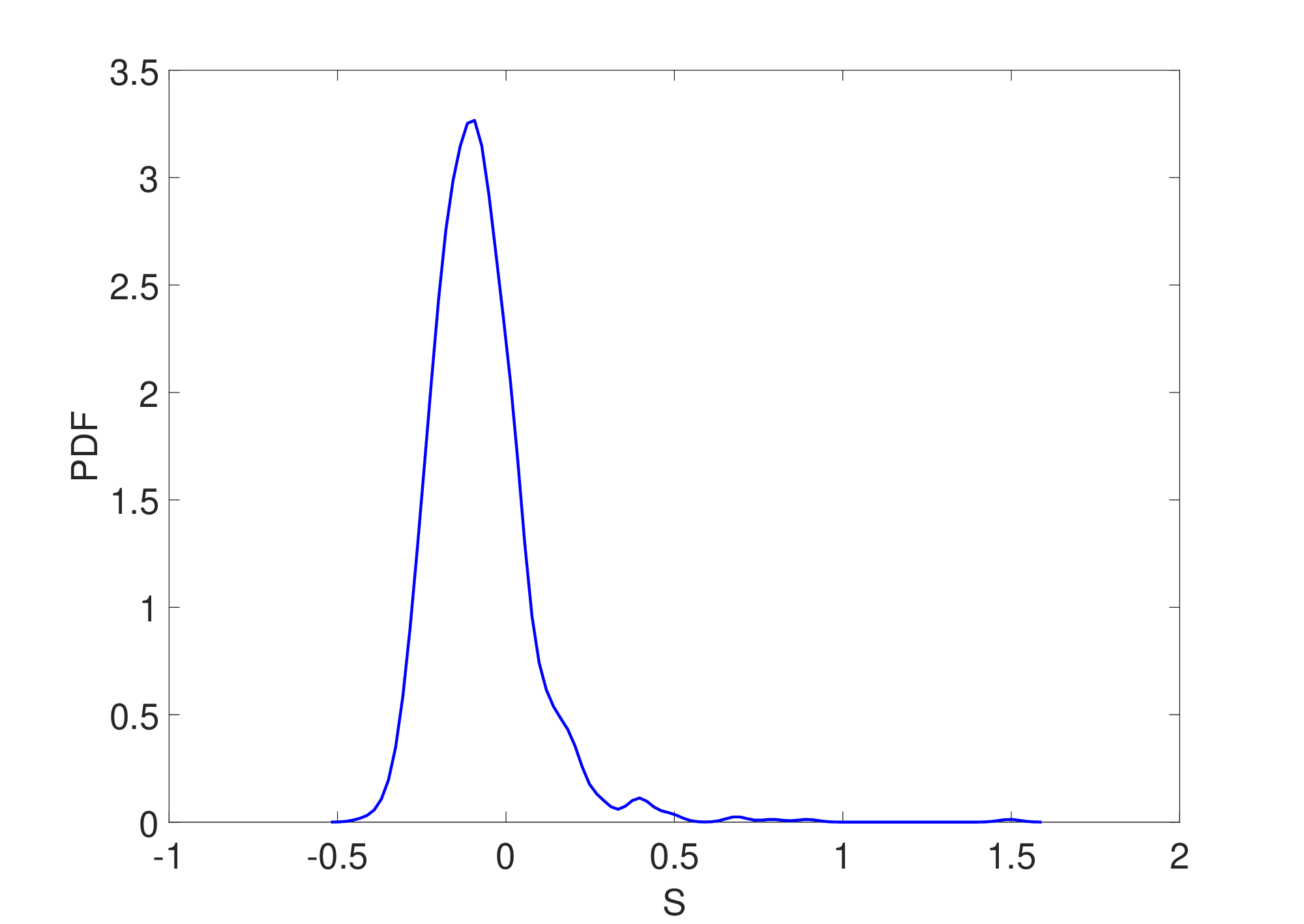}}
\caption{Performance of Sample Approximation with $t$ Distribution}\label{Fig:test2estimator}
\end{figure}

\section{Conclusion}\label{sec:con}

In this paper, we have investigated the statistical properties for various mean-variance portfolio optimization models when using sample approximations. We have made an attempt to provide some theoretical analysis to the robustness and variability of the portfolio optimization model. Essentially, our investigation shows that the sample approximation of the mean-variance model typically has a convergence rate of the square root of the sample size, which is the conventional convergence rate of the Monte Carlo simulation, and the asymptotic variance depends on up to the fourth moment. The results shed some light on why the empirical performance of the mean-variance might exhibit high variation. They may also help quantify the sample error and build confidence region for the portfolio performance in the data driven environment. In this paper, we have also studied the mean-variance DRO models with ambiguity sets constructed by the $\phi$-divergence and the Wasserstein distance, built tractable reformulations and linked the results to the literature. The study shows these mean-variance DRO models can be addressed under a unified expectation-based stochastic optimization framework.

\section*{Acknowledgements}

This paper was previously uploaded to ResearchGate in October 2025 (\url{https://www.researchgate.net/publication/396442050_Mean-Variance_Portfolio_Selection_Statistical_Properties_and_Distributionally_Robust_Optimization}). The work was also presented on The First Youth Academic Conference of the Operations Research Society of China in August 2025.

\section*{References}

\begin{hangref}

\item Ban, G.-Y., N. El Karoui, A. E. B. Lim. 2018. Machine learning and portfolio optimization,
{\it Management Science}, {\bf 64}(3) 1136-1154.

\item Ben-Tal, A., D. den Hertog, A. M. B. de Waegenaere, B. Melenberg, G.
Rennen. 2013. Robust solutions of optimization problems affected by
uncertain probabilities. {\it Management Science}, {\bf 59}(2)
341-357.

\item Blanchet, J., L. Chen, X. Y. Zhou. 2022. Distributionally robust mean-variance portfolio selection with Wasserstein distances. {\it Management Science}, {\bf68}(9) 6382-6410.

\item Blanchet, J., Y. Kang, K. Murthy. 2019. Robust Wasserstein profile inference and applications to machine learning. {\it Journal of Applied Probability}, {\bf 56}(3) 830-857.

\item Britten-Jones, M. 1999. The sampling error in estimates of mean-variance efficient portfolio weights. {\it The Journal of Finance}, {\bf 54}(2) 655-671.
    
\item Burke, J. 2020. Markowitz mean-variance portfolio theory. Available at \url{http://sites.math.washington.edu/~burke/crs/408/notes/fin/mark1.pdf}.  

\item Boyd, S., L. Vandenberghe. 2004. {\it Convex Optimization},
Cambridge University Press, Cambridge, UK.

\item Brodie, J., I. Daubechies, C. De Mol, D. Giannone, I. Loris. 2009. Sparse and stable Markowitz portfolios. {\it Proceedings of the National Academy of Sciences}, {\bf106}(30) 12267-12272.

\item DeMiguel, V., L. Garlappi, F. J. Nogales, R. Uppal. 2009a. A generalized approach to portfolio optimization: Improving performance by constraining portfolio norms. {\it Management Science}, {\bf55}(5) 798-812.

\item DeMiguel, V., L. Garlappi, R. Uppal. 2009b. Optimal versus naive diversification: How inefficient is the 1/N portfolio strategy? {\it The Review of Financial Studies}, {\bf22}(5) 1915-1953.

\item Grant, M., S. Boyd. 2020. CVX: Matlab software for disciplined convex programming. 

\item Hong, L. J., Z. Hu, L. Zhang. 2014. Conditional value-at-risk approximation to
value-at-risk constrained programs: A remedy via Monte Carlo. {\it
INFORMS Journal on Computing}, {\bf 26}(2) 385-400.

\item Hu, Z., L. J. Hong. 2012. Kullback-Leibler divergence constrained
distributionally robust optimization. {\it Optimization Online}.
\url{http://www.optimization-online.org/DB_FILE/2012/11/3677.pdf}.

\item Hu, Z., L. J. Hong. 2022. Robust simulation with likelihood-ratio constrained input uncertainty. {\it INFORMS Journal
	on Computing}, {\bf 34}(4) 2350-2367.

\item Hu, Z., D. Zhang. 2018. Utility‐based shortfall risk: Efficient computations via Monte Carlo. {\it Naval Research Logistics}, {\bf 65}(5) 378-392.

\item Huber, P. J. 1964. Robust estimation of a location parameter. {\it The Annals of Mathematical Statistics}, {\bf 35}(1) 73-101.  

\item Jagannathan, R., T. Ma. 2003. Risk reduction in large portfolios: Why imposing the wrong constraints helps. {\it The Journal of Finance}, {\bf 58} 1651-1684.

\item Kuhn, D., S. Shafiee, W. Wiesemann. 2025. Distributionally robust optimization. {\it Acta Numerica}, {\bf 34} 579-804.

\item Lim, A. E., J. G. Shanthikumar, G. Y. Vahn. 2011. Conditional value at risk in portfolio optimization: Coherent but fragile. {\it Operations Research Letters}, {\bf39}(3) 163-171.

\item Markowitz H. 1952. Portfolio selection. {\it The Journal of Finance}, {\bf 7}(1):77-91.

\item Pardo, L. 2006. {\it Statistical Inference Based on Divergence
Measures}, Chapman \& Hall/CRC, Boca Raton, Florida.

\item Shapiro, A., D. Dentcheva, A. Ruszczy\'nski. 2014. {\it Lectures on
Stochastic Programming: Modeling and Theory}, Second Edition. SIAM,
Philadelphia.

\item Wang, S., G. Cai, P. Yu, G. Liu, J. Luo. 2023. Mean-variance portfolio optimization with nonlinear derivative securities. {\it Proceedings
of the 2023 Winter Simulation Conference}, 576-587.

\end{hangref}

\end{document}